\documentclass[11pt]{article}

\usepackage{fullpage}
\usepackage{amsmath, amsthm, amsfonts, amssymb, amstext, mathrsfs, enumerate}
\usepackage{mathtools}
\usepackage{graphicx, ragged2e, lscape, framed, xcolor}
\usepackage{subfiles}
\usepackage{booktabs}

\usepackage{pgf, tikz, float, subcaption}
\usetikzlibrary{graphs, graphs.standard, decorations.pathreplacing, calligraphy}

\theoremstyle{plain}
\newtheorem{theorem}{Theorem}[section]
\newtheorem{lemma}[theorem]{Lemma}

\newtheorem{conjecture}[theorem]{Conjecture}
\newtheorem{corollary}[theorem]{Corollary}
\newtheorem{problem}[theorem]{Problem}

\numberwithin{equation}{section}
\allowdisplaybreaks

\newcommand{\affl}[3]{\noindent #1, Email: {\tt #2}\\ \textsc{#3}\\[1.5pt]}
\newcommand{\defeq}{\vcentcolon=}

\usepackage[pagebackref]{hyperref}
\hypersetup{
	colorlinks=true,
    urlcolor=purple,
	linkcolor=purple,
    citecolor=purple,
}

\DeclareMathOperator{\NEPS}{NEPS}

\DeclareMathOperator{\inertia}{In}

\DeclareMathOperator{\rank}{rank}

\def\1{{\bf 1}}

\title{\textbf{Resolution of a problem of Mohar on non-positive inertia}}
\author{Clive Elphick, Hitesh Kumar, Shivaramakrishna Pragada, Thom\'{a}s Jung Spier}
\date{}

\begin{document}
\maketitle
\begin{abstract}
For a graph $G$ of order $n$, its \emph{positive, negative} and \emph{non-positive inertia} is the number of positive, negative and non-positive eigenvalues of its adjacency matrix $A(G)$, respectively. Mohar asked whether every graph with $k$ non-positive eigenvalues has order $O(k^2)$ as $k\to \infty$. Using NEPS, we construct a sequence of non-singular graphs with negative inertia $k$ and order $\Omega(k^{\frac{7}{3}})$ as $k\to \infty$, thus resolving Mohar's problem. Our result also strongly refutes a recent conjecture of Akbari, Elphick, Kumar, Pragada, and Tang involving positive and negative inertia. 
\end{abstract}

\noindent
\textbf{Keywords:} Inertia, Adjacency, Tensor product, NEPS

\noindent
\textbf{MSC2020:} 05C50

\section{Introduction}

Let $G$ be a graph with $n$ vertices and adjacency matrix $A(G)$. We denote the eigenvalues of $A(G)$ by 
\[\lambda_1(G) \ge \lambda_2(G) \ge \ldots \ge \lambda_n(G).\] 
Let $n^+(G)$, $n^0(G)$, and $n^-(G)$ denote the number of positive, zero, and negative eigenvalues of $A(G)$, counting multiplicities, respectively. The \emph{inertia} of $G$, denoted $\inertia(G)$, is the ordered triple $(n^+(G), n^0(G), n^-(G))$. The \emph{rank} of $G$, denoted $\rank(G)$, is defined to be the rank of $A(G)$. 

Akbari, Elphick, Kumar, Pragada, and Tang \cite{Akbari_Elphick_Kumar_Pragada_Tang_2026} proposed the following conjecture concerning graph inertia, which relates to various other results and conjectures concerning order, rank, and inertia of graphs. In particular, this conjecture generalizes to all graphs the well-known \emph{absolute bound} for strongly regular graphs (see \cite{Seidel_1979} and \cite[Chapter 10]{Brouwer_Haemers_2012}). Conjecture \ref{conj:inertia_main} is also a step towards Torga\v{s}ev's problem of estimating $n^+$ in terms of $n^-$ \cite{Torgasev_number_of_positive_negative_1992} (also see \cite{Torgasev_1985, Torgasev_two_neg_1985, Torgasev_three_neg_1989}). We refer the reader to \cite{Akbari_Elphick_Kumar_Pragada_Tang_2026} for detailed motivation.

\begin{conjecture}[\cite{Akbari_Elphick_Kumar_Pragada_Tang_2026}]\label{conj:inertia_main} For any graph $G$, 
    \[ n^+(G) \le {n^-(G)  +1 \choose 2}.\]
\end{conjecture}

In \cite{Akbari_Elphick_Kumar_Pragada_Tang_2026}, the authors provided evidence for Conjecture \ref{conj:inertia_main} by proving it for various graph classes, including random, subquartic, planar, and line graphs. As discussed in \cite{Akbari_Elphick_Kumar_Pragada_Tang_2026}, the only known upper bound for $n^+$ is exponential in $n^-$, whereas the conjectured bound is quadratic in $n^-$. 

Recently, it came to our attention\footnote{We thank Varun Sivashankar and Quanyu Tang for bringing this to our attention.} that Conjecture \ref{conj:inertia_main} is false for a family of graphs previously discussed in a paper by Charles, Farber, Johnson, and Kennedy-Shaffer \cite{Charles_Farber_Johnson_Shaffer_2013}. 
This family of graphs was rediscovered in a recent paper by Chen and Li \cite{Chen_Li_2026}. 

Let $K(k,2)$ denote the Kneser graph whose vertex set is the set of all $2$-subsets of $\{1, \ldots, k\}$ and two vertices are adjacent if and only if they are disjoint. Let $K_k$ denote a complete graph on the vertex set $\{1, \ldots, k\}$. Let $W(k)$ denote the graph obtained from the disjoint union of $K_k$ and $K(k,2)$ by joining a vertex $j\in V(K_k)$ to a vertex $A\in V(K(k,2))$ whenever $j\in A$. It is clear that $W(k)$ has order $k + \binom{k}{2}$. 

\begin{theorem}[\cite{Charles_Farber_Johnson_Shaffer_2013, Chen_Li_2026}] For all $k\ge 5$, $W(k)$ has inertia $\left(\binom{k}{2} +1, 0 , k-1\right)$.     
\end{theorem}

It is clear that for the graphs $W(k)$, we have $n^+ = 1 + {n^- + 1 \choose 2}$, thus showing that Conjecture \ref{conj:inertia_main} is off by additive constant +1! This still leaves open the possibility that the following conjecture may hold. 

\begin{conjecture}\label{conj:negative_inertia_square} The positive inertia of a graph with $k$ negative eigenvalues is at most $O(k^2)$ as $k\to \infty$.
\end{conjecture}

The above conjecture, if true, would also answer affirmatively the following (unpublished) problem of Mohar \cite{Mohar_lecture} (cf. \cite{Mohammadian_2022}).  

\begin{problem}[Mohar \cite{Mohar_lecture}]\label{prob:Mohar} Is it true that the order of a graph with $k$ non-positive eigenvalues is at most $O(k^2)$ as $k\to \infty$. 
\end{problem}

Observe that Problem \ref{prob:Mohar} and Conjecture \ref{conj:negative_inertia_square} are the same problem for non-singular graphs. Even in the general case, it turns out that Problem \ref{prob:Mohar} and Conjecture \ref{conj:negative_inertia_square} are essentially equivalent problems (see Section \ref{section:negative_vs_non_positive}). Our main result in this paper is the following. 

\begin{theorem}\label{thm:inertia_counter} There exists a sequence of non-singular graphs $G_m$ $(m\in \mathbb{N})$ such that 
\[n^+(G_m)=\Theta(n^-(G_m)^{\frac{7}{3}})\]
as $m\to \infty$. 
\end{theorem}

Thus, Conjecture \ref{conj:negative_inertia_square} is strongly false and Mohar's problem \ref{prob:Mohar} does not have a positive answer. In fact, by taking the disjoint union of $G_m$ and appropriately many $K_2$'s, we also have the following corollary. 

\begin{corollary}\label{cor:continuous_exponent}
For any real number $s\in [1, 7/3]$, there exists a sequence of non-singular graphs $G_m$ $(m\in \mathbb{N})$ such that 
\[n^+(G_m)=\Theta(n^-(G_m)^{s})\]
as $m\to \infty$. 
\end{corollary}

To prove Theorem \ref{thm:inertia_counter}, we use a graph operation called Non-complete Extended P-Sum (NEPS) introduced by Cvetkovi\'{c} \cite{Cvetkovic_1971} (cf. \cite{Cvetkovic_Simic_1993, Cvetkovic_Doob_Sachs_1995}). This operation is well-designed from a spectral viewpoint, as its adjacency matrix and eigenvalues can be easily determined. NEPS have been used to disprove several important conjectures. For instance, Razborov \cite{Razborov_1992}, and Cioab\u{a} and Tait \cite{Cioaba_Tait_2011} used NEPS to construct counterexamples for the Alon--Saks--Seymour Conjecture and the Rank-Coloring Conjecture. To construct the sequence of graphs $G_m$, we take NEPS of copies of the complete graph $K_m$ with respect to a basis that is defined using a well-designed polynomial, which we call $(q,s)$-admissible polynomials. Essentially, the $(q,s)$-admissible polynomials provide a convenient computational framework to construct graphs with $n^+ \approx (n^-)^{\frac{q}{s}}$ (see Theorem \ref{thm:admissible_polynomials}). So to prove our Theorem \ref{thm:inertia_counter}, we only need to find a $(7,3)$-admissible polynomial, which we do in Theorem \ref{thm:7_3_admissible}. To improve our result, one only needs to find a $(q,s)$-admissible polynomial with a larger ratio $\frac{q}{s}$. We believe that the ideas developed may find further application to other problems related to inertia and rank. 

We discuss $(q,s)$-admissible polynomials and NEPS in Section \ref{section:large_positive_inertia} along with a proof of Theorem \ref{thm:inertia_counter}. We conclude with some remarks in Section \ref{section:conclusion}. 

\section{Negative inertia vs non-positive inertia}
\label{section:negative_vs_non_positive}

For completeness, here we will show that Conjecture \ref{conj:negative_inertia_square} and Problem \ref{prob:Mohar} are essentially the same questions. 

We first recall two results from matrix theory. 

\begin{theorem}[Interlacing Theorem] \label{thm:interlacing}
Let $A$ be an $n\times n$ Hermitian matrix with eigenvalues $\lambda_1 \ge \lambda_2 \ge \cdots \ge \lambda_n$. Let $B$ be an $m\times m$ principal submatrix of $A$ with eigenvalues $\theta_1\ge \theta_2\ge \cdots \ge \theta_m$. Then
\[\lambda_i \ge \theta_i \ge \lambda_{i+n-m}.\]
\end{theorem}

\begin{lemma}[{\cite[Theorem 1.1]{Barrett_Butler_Catral_Fallat_Hall_Hogben_vanden_Young_2014}}]\label{lemma:principal_submatrix_rank_r}
Any real symmetric matrix with rank $r$ contains a non-singular principal submatrix of order $r$.
\end{lemma}

The following is a consequence of Lemma \ref{lemma:principal_submatrix_rank_r} combined with the Interlacing Theorem that allows us to reduce the problem for general graphs to non-singular graphs. 

\begin{lemma}\label{lemma:induced_subgraph_same_rank}
Every graph $G$ contains an induced subgraph $H$ such that 
    \[n^+(H) = n^+(G),\quad n^-(H)=n^-(G), \quad |V(H)| = \rank(G).\]
\end{lemma}
\begin{proof} By Lemma~\ref{lemma:principal_submatrix_rank_r}, there exists non-singular induced subgraph $H$ of $G$ with $\rank(H)=\rank(G)$. Note that, by Theorem~\ref{thm:interlacing}, $n^+(G)\geq n^+(H)$ and $n^-(G)\geq n^-(H)$, and therefore
\[
\rank(H) = \rank(G) = n^+(G) + n^-(G) \geq n^+(H) + n^-(H) = \rank(H),
\]
from which follows that $n^+(H)=n^+(G)$ and $n^-(H)=n^-(G)$. Moreover, since $H$ is non-singular, 
\[|V(H)|=\rank(H)=\rank(G).\qedhere\] 
\end{proof}

Now, for given $k\in \mathbb{N}$, define 
\[ P(k) := \max_{G}\{n^+(G): n^-(G) \le  k\}\quad  \text{and}\quad M(k):=\max_{G}\{n: n^0(G)+n^-(G)\le k\}.\]
If $G$ is a graph with $n^-(G) + n^0(G) \le k$ and order $M(k)$, then 
\[ M(k) = n^+(G) + n^0(G) + n^-(G)\le P(k) + k.\]
Conversely, if $G$ is a graph with $n^-(G)\le k$ and $n^+(G)=P(k)$, then by Lemma \ref{lemma:induced_subgraph_same_rank}, $G$ has an induced subgraph $H$ with same positive and negative inertia. Add $k-n^-(G)$ isolated vertices to $H$ and call the new graph $H'$. Then $n^0(H') + n^-(H') = (k-n^-(G)) + n^-(G) = k$ and $|V(H')|=P(k)+k$. Thus, 
\[P(k) + k \le M(k).\]
We conclude that 
\[ M(k) = P(k)+k.\]
In particular, 
\[P(k) = O(k^2) \quad \iff\quad  M(k)=O(k^2) \quad \text{as}\quad k\to \infty.\]
Thus, Conjecture \ref{conj:negative_inertia_square} has a positive answer if and only if Problem \ref{prob:Mohar} does. 

\section{Graphs with large positive inertia}
\label{section:large_positive_inertia}

\subsection{$(q,s)$-admissible polynomials}

Let $s,q\in \mathbb{N}$ with $s\leq q$. Consider a multivariate polynomial
\[P(x_1,\dots, x_q)=\sum_{T\subseteq [q]} c_{T} \prod_{i\in T}x_i\, \in \mathbb{Z}[x_1,\dots,x_q].\]
We say that $P$ is a \emph{Boolean polynomial} if $P(x_1, \ldots, x_q)\in \{0,1\}$ for all $(x_1, \ldots, x_q)\in \{0,1\}^q$. We say that a Boolean polynomial $P$ is \emph{$(q,s)$--admissible} if all of the following hold:
\begin{equation}
    P(1,\dots, 1)=0, \qquad c_{[q]}>0, \qquad  s=\max\{|T|: T\subseteq [q],  c_T<0\}.
\end{equation}

The following result allows us to construct graphs with large positive inertia relative to their negative inertia.  

\begin{theorem}\label{thm:admissible_polynomials}
Given a $(q,s)$-admissible polynomial $P$, there exists a sequence of graphs $G_m$ $(m\in \mathbb{N})$ of order $m^q$ such that 
\[ n^+(G_m) = \Theta(m^q), \quad n^-(G_m)=\Theta(m^s)\quad \text{and}\quad n^0(G_m)=0\quad \text{as}\quad m\to \infty.
\]
\end{theorem}

In particular, if there exists a $(q,s)$-admissible polynomial, then there exists a sequence of graphs with $n^+ = \Theta((n^-)^{\frac{q}{s}})$. So our problem boils down to finding $(q,s)$-admissible polynomials such that the ratio $\frac{q}{s}$ is large. This is much more convenient computationally. We will prove Theorem \ref{thm:admissible_polynomials} in the next Subsection \ref{subsection:NEPS}.

Let us here show the existence of a $(7,3)$-admissible polynomial which implies our main result Theorem \ref{thm:inertia_counter} in light of the above Theorem \ref{thm:admissible_polynomials}. 

\begin{theorem}\label{thm:7_3_admissible}
Let
\[
a\defeq x_1 + x_2 - x_1 x_2,\qquad b\defeq x_3 + x_4 - x_3 x_4,\qquad
c\defeq x_1 x_2 x_3 x_4,
\]
and define
\[
P\defeq a b-c+x_5 x_6 (1-a)+x_5 x_7 (1-b)
-x_5 x_6 x_7(1-c)\in \mathbb{Z}[x_1, \ldots, x_7].
\]
Then $P$ is $(7,3)$--admissible.
\end{theorem}
\begin{proof}
Note that if $x_i\in \{0,1\}$ for every $i$, then $a,b,c\in \{0,1\}$. We partition the possible cases according to the values of $a$, $b$ and $c$. Observe that $c=1$ implies $a=1=b$. Hence, the possible cases are 
\[
P(x_1,\dots, x_7)=
\begin{cases}
0&\text{ if }a=b=c=1;\\
1-x_5x_6x_7&\text{ if }a=b=1,\ c=0;\\
x_5x_6(1-x_7)&\text{ if }a=0,\ b=1,\ c=0;\\
x_5x_7(1-x_6)&\text{ if }a=1,\ b=c=0;\\
x_5(x_6 + x_7- x_6 x_7)&\text{ if }a=b=c=0.
\end{cases}
\]
In all cases, if $(x_1,\dots, x_7)\in \{0,1\}^7$, then $P(x_1,\dots, x_7)\in \{0,1\}$, and so $P$ is Boolean. Clearly,  $P(1,\dots, 1)=0$.

Now, the multilinear expansion of the polynomial $P$ is
\[
\begin{aligned}
P= & \,\, x_1 x_3+x_2 x_3+x_1 x_4+x_2 x_4+x_5 x_6+x_5 x_7\\
&-x_1 x_2 x_3 -x_1 x_2 x_4-x_1 x_3 x_4-x_2 x_3 x_4 -x_1 x_5 x_6-x_2 x_5 x_6-x_3 x_5 x_7-x_4 x_5 x_7-x_5 x_6 x_7\\
&+x_1 x_2 x_5 x_6+x_3 x_4 x_5 x_7+x_1 x_2 x_3 x_4 x_5 x_6 x_7.
\end{aligned}
\]
Thus, $c_{[7]}=1>0$ and $\max\{|T|: T\subseteq [7],  c_T<0\}=3$. Hence, $P$ is $(7,3)$-admissible.
\end{proof}

\subsection{NEPS and Proof of Theorem \ref{thm:admissible_polynomials}}
\label{subsection:NEPS}

Let us first recall the NEPS operation; for further details on this operation, refer \cite{Cvetkovic_1971, Cvetkovic_Simic_1993, Cvetkovic_Doob_Sachs_1995}. Given $q\in \mathbb{N}$, $\mathcal{B}\subseteq \{0,1\}^q\setminus \{(0, \ldots, 0)\}$ and graphs $G_1, \ldots, G_q$, the graph denoted by
\[ \NEPS(G_1, \ldots, G_q; \mathcal{B})\]
is the graph whose vertex set is the Cartesian product of the vertex sets of $G_i$'s and two vertices $(u_1, \ldots, u_q)$ and $(v_1, \ldots, v_q)$ are adjacent if and only if there exists $(b_1, \ldots, b_q)\in \mathcal{B}$ such that $u_i = v_i$ if $b_i=0$ and $u_iv_i\in E(G_i)$ if $b_i=1$ for all $1\le i\le q$. The set $\mathcal{B}$ is called the \emph{basis} of $\NEPS(G_1, \ldots, G_q; \mathcal{B})$.  
It is known that the adjacency matrix of this NEPS graph is given by (see \cite[Section 2.5]{Cvetkovic_Doob_Sachs_1995})
\begin{equation} \label{eq:NEPS_adjacency}
    \sum_{(b_1, \ldots, b_q)\in \mathcal{B}} A(G_1)^{b_1}\otimes \cdots \otimes A(G_q)^{b_q},
\end{equation}
where $\otimes$ denotes the tensor product. 

For a subset $S\subseteq [q]$, let $1_S\in \{0,1\}^q$ denote the indicator vector of $S$. Now, given a $(q,s)$-admissible polynomial
\begin{equation}\label{eq:P_F_defn}
P(x_1,\dots, x_q)=\sum_{T\subseteq [q]} c_{T} \prod_{i\in T}x_i,
\end{equation}
define 
\[\mathcal{F}_P:=\{S\subseteq [q]: P(1_{[q]\setminus S}) = 1 \},\]
where $[q]\setminus S$ denotes the complement of $S$ in $[q]$. Let
\[ \mathcal{B}_P:=\left\{1_S: S\in \mathcal{F}_P\right\}.\]
Note that $(0,\ldots, 0)\notin \mathcal{B}_P$ since $P(1_{[q]}) = P(1,\ldots, 1) = 0$. For sufficiently large $m\in \mathbb{N}$ and the basis $\mathcal{B}_P$ corresponding to $P$, define 
\[ G_m:= \NEPS(\underbrace{K_m, \ldots, K_m}_{q}; \mathcal{B}_P),\]
where $K_m$ denotes the complete graph of order $m$. We will now argue that the graphs $G_m$ $(m\in \mathbb{N})$ have the desired inertia, thus proving Theorem \ref{thm:admissible_polynomials}. 

We first determine the adjacency matrix of $G_m$. Note that $A(K_m) = J_m - I_m$, where $J_m$ and $I_m$ denote the all-ones and the identity matrix of order $m$. For a subset $S\subseteq [q]$ consider the matrix
\[
A_S \defeq \bigotimes_{i=1}^{q} M_i(S), \quad \text{where}\quad M_i(S)\defeq 
\begin{cases} J_m - I_m & \text{if }i\in S; \\ 
I_m & \text{otherwise.}
\end{cases} 
\]
Using the definitions of $\mathcal{F}_P$, $\mathcal{B}_P$ and \eqref{eq:NEPS_adjacency}, the following lemma is immediate.  

\begin{lemma}\label{lemma:adj_matrix_G_m} The adjacency matrix of $G_m$ is given by
\[
A(G_m) = \sum_{S\in \mathcal{F}_P} A_S = \sum_{S\in \mathcal{F}_P} \bigotimes_{i=1}^q M_i(S).
\]
\end{lemma}

The eigenvalues of a NEPS graph can be determined using \eqref{eq:NEPS_adjacency} as is done in Theorem 2.23 of the book by Cvetkovi\'{c}, Doob and Sachs \cite{Cvetkovic_Doob_Sachs_1995}. We find an expression for the eigenvalues of $G_m$ that is more amenable for our calculations.  

Let
\[
\mathbb{R}^m = U_0\oplus U_1,\quad \text{where}\quad U_0\defeq \langle 1\rangle,\quad U_1\defeq \langle 1\rangle^{\perp},
\]
and for $T\subseteq [q]$ define
\[
\mathcal{U}_T \defeq \bigotimes_{i=1}^q U_{1_T(i)},\quad \text{where}\quad 1_T(i)\defeq \begin{cases}
1 & \text{ if }i\in T;\\
0 & \text{ otherwise.}
\end{cases}
\]

Observe that, if $T, T'\subseteq [q]$ and $T\neq T'$, then $\mathcal{U}_{T}$ and $\mathcal{U}_{T'}$ are orthogonal (w.r.t. the usual inner product). More than that, it holds
\[
\mathbb{R}^{m^q}\simeq \underbrace{\mathbb{R}^m\otimes \cdots \otimes \mathbb{R}^m}_{q}= \bigoplus_{T\subseteq [q]} \mathcal{U}_T.
\]

\begin{lemma}\label{lemma:eigenvalues_G_m} If $T\subseteq [q]$, then $\mathcal{U}_T$ is an eigenspace of $A(G_m)$ with dimension $(m-1)^{|T|}$ and corresponding eigenvalue 
\[
\lambda_T(G_m) \defeq \sum_{S\in \mathcal{F}_P} (-1)^{|S\cap T|}(m-1)^{|S\setminus T|}. 
\]
\end{lemma}

\begin{proof} Since $\mathcal{U}_T=\bigotimes_{i=1}^q U_{1_T(i)}$, we have that $\mathcal{U}_T$ is a tensor product of $q-|T|$ vector spaces of dimension $1$ and $|T|$ vector spaces of dimension $m-1$, from which follows that its dimension is $(m-1)^{|T|}$.

Now, note that $J_m-I_m$ acts on $U_0$ and $U_1$ as the scalars $m-1$ and $-1$, respectively. As a consequence if $i \in S$, then $M_i(S)$ acts on $U_0$ and $U_1$ as the scalars $m-1$ and $-1$, respectively, and if $i\notin S$, then $M_i(S)$ acts on both $U_0$ and $U_1$ as the scalar $1$. It follows that $M_i(S)$ acts on $U_{1_T(i)}$ as the scalar $m-1$, $-1$ or $1$ if $i\in S\setminus T$, $i\in S\cap T$ or $i\notin S$, respectively. 

As $A_S=\bigotimes_{i=1}^{q} M_i(S)$, we obtain that $A_S$ acts on $\mathcal{U}_T =\bigotimes_{i=1}^q U_{1_T(i)}$ as the scalar $(-1)^{|S\cap T|}(m-1)^{|S\setminus T|}$. Since $A(G_m)=\sum_{S\in \mathcal{F}_P} A_S$ and each $A_S$ acts on $\mathcal{U}_T$ as the scalar $(-1)^{|S\cap T|}(m-1)^{|S\setminus T|}$. We conclude that $\mathcal{U}_T$ is an eigenspace of $A(G_m)$ with corresponding eigenvalue $\sum_{S\in \mathcal{F}_P} (-1)^{|S\cap T|}(m-1)^{|S\setminus T|}$. 
\end{proof}

The next lemma shows that the eigenvalues of $G_m$ can be expressed as a polynomial in $m$ with coefficients coming from the coefficients of the polynomial $P$. 

\begin{lemma}\label{lemma:pol_eigen} For $T\subseteq [q]$,
\[
\lambda_T(G_m) = \sum_{T\subseteq R \subseteq [q]} c_R\, m^{q-|R|},
\]
where $c_R$'s are as defined in \eqref{eq:P_F_defn}. 
\end{lemma}

\begin{proof} From the definition of $\mathcal{F}_P$ it is clear that 
\[ P(x_1, \ldots, x_q) = \sum_{S\in \mathcal{F}_P} \prod_{i\in S}(1-x_i)\prod_{i\notin S}x_i.\]

For $T\subseteq [q]$, let $\partial_T\defeq \prod_{i\in T}\partial_{x_i}$, where $\partial_{x_i}$ denotes the partial derivative w.r.t. $x_i$. Note that,
\begin{align*}
\partial_T P(x_1,\dots, x_q) 
&=\partial_T \left(\sum_{S\in \mathcal{F}_P} \prod_{i\in S}(1-x_i)\prod_{i\notin S}x_i\right)\\
&=\sum_{S\in \mathcal{F}_P}\partial_T \left( \prod_{i\in S}(1-x_i)\prod_{i\notin S}x_i\right)\\
&=\sum_{S\in \mathcal{F}_P}(-1)^{|T\cap S|} \prod_{i\in S\setminus T}(1-x_i)\prod_{i\notin S\cup T}x_i,
\end{align*}
and thus
\begin{align}\label{eq:partial_1}
\partial_T P\left(\frac{1}{m},\dots, \frac{1}{m}\right) 
&=\sum_{S\in \mathcal{F}_P}(-1)^{|T\cap S|} \prod_{i\in S\setminus T}\left(1-\frac{1}{m}\right)\prod_{i\notin S\cup T}\frac{1}{m}\nonumber\\
&=\sum_{S\in \mathcal{F}_P}(-1)^{|T\cap S|}\, \frac{(m-1)^{|S\setminus T|}}{m^{|S\setminus T|}}\, \frac{1}{m^{|[q]\setminus (S\cup T)|}}\nonumber\\
&=\sum_{S\in \mathcal{F}_P}(-1)^{|T\cap S|} \frac{(m-1)^{|S\setminus T|}}{m^{|[q]\setminus T|}}\nonumber\\
&=\frac{1}{m^{q-|T|}}\, \lambda_T(G_m).
\end{align}

On the other hand,
\begin{align*}
\partial_T P(x_1,\cdots, x_q)
&=\partial_T\left(\sum_{R\subseteq [q]} c_R \prod_{i\in R} x_i\right)\\ &=\sum_{R\subseteq [q]} c_R \partial_T\left(\prod_{i\in R} x_i\right)\\ &=\sum_{T\subseteq R\subseteq [q]} c_R\, \prod_{i\in R\setminus T} x_i,
\end{align*}
from which it follows that
\begin{align}\label{eq:partial_2}
\partial_T P\left(\frac{1}{m},\dots, \frac{1}{m}\right) &=\sum_{T\subseteq R\subseteq [q]} c_R\, \prod_{i\in R\setminus T} \frac{1}{m}\nonumber\\
&=\sum_{T\subseteq R\subseteq [q]} c_R\, \frac{1}{m^{|R\setminus T|}}\nonumber\\
&=\frac{1}{m^{q-|T|}}\sum_{T\subseteq R\subseteq [q]} c_R m^{q-|R|}.
\end{align}

The assertion follows from \eqref{eq:partial_1} and \eqref{eq:partial_2}.
\end{proof}

Lemma~\ref{lemma:pol_eigen} implies that for $T\subseteq [q]$ and for $m\in \mathbb{N}$ sufficiently large, the sign of $\lambda_T(G_m)$ is controlled by the coefficients $c_R$ for $T\subseteq R\subseteq [q]$. The next lemmas show that the properties used to define a $(q,s)$-admissible polynomial essentially govern the inertia of $G_m$. 

\begin{lemma} As $m\to \infty$, 
    \[n^+(G_m) = \Theta(m^q).\]
\end{lemma}

\begin{proof}
Since $G_m$ has order $m^q$, we have $n^+(G_m)\leq m^q$. By Lemmas~\ref{lemma:eigenvalues_G_m} and~\ref{lemma:pol_eigen}, we have $\lambda_{[q]}(G_m) = c_{[q]}>0$ with multiplicity $\dim \mathcal{U}_{[q]}=(m-1)^q$. Therefore, $n^+(G_m)\geq (m-1)^q$. Assertion follows. 
\end{proof}

\begin{lemma} As $m\to \infty$, 
    \[n^-(G_m)=\Theta(m^s).\]
\end{lemma}

\begin{proof}
Consider $T\subseteq [q]$ with $|T|=s$ and $c_T<0$. Then, by Lemmas~\ref{lemma:eigenvalues_G_m} and~\ref{lemma:pol_eigen}, 
\[
\lambda_T(G_m) = \sum_{T\subseteq R \subseteq [q]} c_R\, m^{q-|R|} \sim c_T m^{q-|T|}<0,
\]
for $m\in \mathbb{N}$ sufficiently large. It follows that
\begin{align*}
    n^-(G_m) \geq \dim \mathcal{U}_T = (m-1)^{|T|}=(m-1)^s.
\end{align*}
Now, note that if $T\subseteq [q]$ with $|T|>s$, then $c_R\geq 0$ for every $T\subseteq R\subseteq [q]$. Then by Lemma~\ref{lemma:pol_eigen},
\[
\lambda_T(G_m) = \sum_{T\subseteq R \subseteq [q]} c_R\, m^{q-|R|}\geq 0.
\]
As a consequence, if $\lambda_T(G_m)<0$, then $|T|\leq s$. It follows that 
\begin{align*}
n^-(G_m) &\leq \sum_{T\subseteq [q],\, |T|\leq s} \dim \mathcal{U}_T \\ &=\sum_{T\subseteq [q],\, |T|\leq s} (m-1)^{|T|} \\ &=\sum_{k=0}^s {q \choose k} (m-1)^k\\ &\leq m^s\sum_{k=0}^s {q \choose k} \\ &=O_{q,s}(m^s). 
\end{align*}
We conclude that $n^-(G_m)=\Theta(m^s)$.
\end{proof}

\begin{lemma} For all sufficiently large $m$, 
    \[n^0(G_m)=0.\]
\end{lemma}

\begin{proof} Consider $T\subseteq [q]$. By Lemmas~\ref{lemma:eigenvalues_G_m} and~\ref{lemma:pol_eigen} and since $c_{[q]}>0$, $\lambda_T(G_m)$ is a polynomial in $m$ with positive constant coefficient. Since $q$ is fixed, the degree of the polynomial $\lambda_T(G_m)$ is bounded, and so for all large enough $m$, $\lambda_T(G_m)\neq 0$. Assertion holds. 
\end{proof}

Theorem \ref{thm:admissible_polynomials} now follows from the above lemmas.

\section{Concluding remarks}
\label{section:conclusion}

Given that there are graphs with $n^- = k$ and $n^+ = \Omega(k^{\frac{7}{3}})$ (Theorem \ref{thm:inertia_counter}), it is natural to wonder what the correct upper bound for $n^+$ is in terms of $n^-$. We also found $(32,12)$-admissible polynomials showing the existence of a graph family with $n^+ = \Omega(k^{\frac{8}{3}})$; we chose not to include them to keep our presentation short and easily verifiable. In \cite{Akbari_Elphick_Kumar_Pragada_Tang_2026}, the authors observe an upper bound for $n^+(G)$ which is exponential in $n^-(G)$ for a given graph $G$. But the following problem still remains challenging. 

\begin{problem}[\cite{Akbari_Elphick_Kumar_Pragada_Tang_2026}]\label{problem:polynomial_bound} Does there exist a polynomial $f$ such that for every graph $G$, 
\[n^+(G)\le f(n^-(G)).\]
\end{problem}

\section*{Acknowledgements}

The authors thank Bojan Mohar for helpful comments. 

\section*{Declaration of AI use}

The authors acknowledge the use of ChatGPT (GPT-5.5, OpenAI; accessed July 2026) solely for preliminary brainstorming and the exploration of possible proof strategies. AI tools were not used to draft the manuscript. All formal statements, arguments, and proofs in the manuscript were written and checked by the authors, who take full responsibility for the accuracy and integrity of the article. 

\bibliographystyle{plain}
\bibliography{references_1.bib}

\vspace{0.4cm}
\affl{Clive Elphick}{clive.elphick@gmail.com}{School of Mathematics, University of Birmingham, Birmingham, UK}

\affl{Hitesh Kumar}{hitesh.kumar.math@gmail.com, hitesh\_kumar@sfu.ca}{Department of Mathematics, Simon Fraser University, Burnaby, Canada}

\affl{Shivaramakrishna Pragada}{shivaramakrishna\_pragada@sfu.ca, shivaramkratos@gmail.com}{Department of Mathematics, Simon Fraser University, Burnaby, Canada}

\affl{Thom\'{a}s Jung Spier}{thomasjung@dcc.ufmg.br}{Department of Computer Science, Universidade Federal de Minas Gerais, Belo Horizonte, Brazil}

\end{document}